\documentclass[reqno, 11pt]{amsart}
\usepackage[mathscr]{eucal}
\usepackage[all]{xy}
\usepackage{mathrsfs}
\usepackage{xypic}
\usepackage{amsfonts}
\usepackage{amsmath}
\usepackage{amsthm,bm}
\usepackage{amssymb,enumerate}
\usepackage{latexsym}
\usepackage{tabularx}
\usepackage{graphicx}
\usepackage{pict2e}
\usepackage{tikz}
\usepackage[pagebackref]{hyperref}
\usepackage[text={160mm,240mm},centering]{geometry}
\usepackage{enumitem}
\usepackage{lineno}
\usepackage{color}
\usepackage{fancyhdr}

\input diagxy
\input xy

\newtheorem{thm}{Theorem}[section]

\newtheorem{cor}[thm]{Corollary}

\newtheorem{lem}[thm]{Lemma}
\newtheorem{prop}[thm]{Proposition}

\theoremstyle{definition}
\newtheorem{defn}[thm]{Definition}
\newtheorem{rem}[thm]{Remark}

\newtheorem{ex}[thm]{Example}
\numberwithin{equation}{section}

\newcommand{\cX}{\mathcal{X}}
\newcommand{\cG}{\mathcal{G}}
\newcommand{\ep}{\varepsilon}
\newcommand{\dbar}{\bar \partial}
\newcommand{\ddbar}{\partial\bar{\partial}}
\newcommand{\h}{\widehat}
\newcommand{\w}{\widetilde}

\newcommand{\ol}{\overline}
\newcommand{\mc}{\mathbb{C}}
\newcommand{\im}{\mathrm{Im}}

\begin{document}


\title{Nakano Positivity of direct image sheaves of adjoint line bundles with mild singularities}

\author{Yongpan Zou}
\address{Graduate School of Mathematical Sciences, The University of Tokyo, 3-8-1 Komaba, Meguro-Ku, Tokyo 153-8914, Japan}
\email{zouyongpan@gmail.com}

\date{\today}


\begin{abstract}
In this paper, we explore the Nakano positivity of direct image sheaves of twisted relative canonical bundles when the metric of the twisted line bundle has mild singularities. We address this problem using two methods: $L^2$ estimates and curvature computations within the framework of $L^2$ Hodge theory.

\end{abstract}

\maketitle
\setcounter{tocdepth}{1}


\section{Introduction}
This note aims to study the Nakano positivity properties of direct image sheaves of relative canonical bundle twisted by holomorphic line bundle with a possible singular metric. More specifically, let $p:\cX\to D$ be a holomorphic proper morphism from a K\"ahler manifold $\cX$ onto the polydisk $D$, and let $L$ be a holomorphic line bundle endowed with a possibly singular hermitian metric $h_L$. In general, we assume that the metric $h_L$ is positively curved, i.e., its curvature current $i\Theta_{h_L}(L)\ge 0$. An important object of study is the direct image sheaf
$$ \mathcal G:=p_*((K_{\cX/D}+L)\otimes \mathcal I(h_L))
$$
endowed with the metric which induced by $h_L$, here $\mathcal I(h_L)$ be the multiplier ideal sheaf associated with the metric $h_L$. 
Given the numerous important applications in algebraic and complex geometry, this topic has attracted significant attention from researchers.

Griffiths positivity and Nakano positivity of holomorphic vector bundles are two major positive concepts in complex geometry. Between the two, Nakano positivity is stronger than Griffiths positivity. In \cite{Bo09}, Berndtsson proved the induced metric on $\mathcal G$ has nonnegative curvature in the sense of Nakano when the metric $h_L$ on $L$ is smooth. Since then, researchers have begun to investigate positivity properties in the singular case. 
Regarding Griffiths positivity, when the metric $h$ is smooth, it is well known that a holomorphic vector bundle $(E, h)$ is Griffiths semi-positive if and only if its dual $(E^{\ast}, h^{\ast})$ is Griffiths semi-negative. 
This equivalence allows for the generalization of Griffiths positivity to the singular case. Consequently, the theory of Griffiths positivity for direct image sheaves of adjoint line bundles with singular metrics has been established (see \cite{PT}, \cite{Paun18}, \cite{HPS18}, \cite{DWZZ18} for various results and generalizations).
Since Griffiths positivity in a singular setting has been thoroughly studied, it is natural to explore Nakano positivity in this context. However, the duality property is no longer valid, which complicates the analysis of Nakano positivity (for one definition of Nakano positivity of singular Hermitian metrics on holomorphic vector bundles, consult\cite{Rau15}). Recently, Cao--Guenancia--P\u{a}un generalized Berndtsson's curvature formulas to cases where the metric $h_L$ has analytic singularities and the base space is $1$-dimensional (\cite{CGP}). By using the horizontal lift, they successfully obtained a representative of the holomorphic section and established the general curvature formula through this representative.

Meanwhile, in \cite{DNWZ20}, the authors provided a characterization of the Nakano positivity of holomorphic vector bundle with smooth metric via the optimal $L^2$-estimate condition. 
Due to the fact that in our main case $Z$ below, the metric $h_L$ of the line bundle is singular, but it is smooth concerning the $t$ variables. It is possible to use the optimal $L^2$-estimate condition to study the Nakano positivity in these types of singular cases, as I introduce the settings of my study below.

\subsection{Set-up of case $Z$} \label{setup2}
Let $p:\cX\to D$ be a holomorphic proper fibration (i.e., submersion) from a $(n+m)$-dimensional K\"ahler manifold $\cX$ onto the bounded pseudoconvex domain $D\subset \mathbb C^{m}$, and let $(L,h_L)$ be a holomorphic line bundle endowed with a possibly singular hermitian metric $h_L$. Let $\Omega\subset \cX$ be a coordinate subset on $\cX$. We take $(z_1,\dots z_n, t _1,\dots, t_m)$ as a coordinate system on $\Omega$ such that the last $m$ variables $t_1,\dots, t_m$ corresponds to the map $p$ itself. 
\begin{enumerate} [label={\color{violet}(Z.\arabic*)}]
\item \label{Z1} The metric $h_L = e^{-\psi_L}$ and the local weights $\psi_L$ have \textbf{Poincar\'e type singularities}, or \textbf{logarithmic type singularities}, or \textbf{klt type singularities} along $E$, as illustrated in Example \ref{3-ex} below.

\item \label{Z2} The Chern curvature of $(L,h_L)$ satisfies
$i\Theta_{h_L}(L):= i\ddbar \psi_L \ge 0$
in the sense of currents on $\cX$.

\item \label{Z3} The multiplier ideal sheaf $\mathcal I(h_{L_t}) = \mathcal{O}_{X_t}$ for each $t\in D$.
\item \label{Z4} The K\"ahler manifold $\cX$ contains a Stein Zariski open subset.
\end{enumerate}
\medskip

\noindent With these assumptions we set
\begin{align} \label{F}
\mathcal F:=p_*((K_{\cX/D}+L)\otimes \mathcal I(h_L)) = p_*(K_{\cX/D}+L).
\end{align}
\noindent By assumption \ref{Z2} and the K\"ahler version of Ohsawa-Takegoshi theorem (cf. \cite{CaoOT}), $\mathcal F$ is indeed a vector bundle and $\mathcal F_t = H^0(X_t, K_{X_t} + L_t)$ for every $t \in D$. There is a Hermitian  metric $\|\cdot \|$ on $\mathcal F$ induced by $h_L$, i.e., for any $u_t \in \mathcal F_t$,
$$ \|u_t\|^2 = \int_{X_t} c_n u\wedge \bar u e^{-\psi_L} ~~\text{with} ~~ c_n=(\sqrt{-1})^{n^2} $$

 By assumption \ref{Z1}, the metric $\|\cdot \|$ on the direct image sheaf $\mathcal F$ is well-defined and smooth. Indeed, we can use partitions of unity to reduce to checking that integrals of the form $\int_{\Omega\cap X_t} |u_t|^2e^{-\psi_L}$ vary smoothly with $t$, where $\psi_L$ is given by the expression in \ref{Z1}. The reader can consult Lemma $2.2$ in \cite{CGP} for more details. We now introduce the main theorem of this paper.
\begin{thm} \label{Thm2}
Under the set-up of case $Z$, the Hermitian holomorphic vector bundle $(\mathcal F, \|\cdot\|)$ over $D$ defined in \eqref{F} is semi-positive in the sense of Nakano.
\end{thm}


\begin{rem}
To solve the $\dbar$-equation with a singular weight on $\cX$, we must assume that the K\"ahler manifold $\cX$ contains a Stein Zariski open subset, which is why assumption \ref{Z4} is required. An important example is when $p: \cX \to D$ is a projective morphism, in which case assumption \ref{Z4} is naturally satisfied.
\end{rem}

\begin{ex} \label{3-ex}
We provide some important examples. We assume that there exists a divisor $E= E_1+\dots + E_N$  whose support is contained in the total space $\cX$ of $p$
such that the following requirements are fulfilled. The divisor $E$ intersects each fiber transversally, i.e., for every $t\in D$ the restriction divisor $E_t:= E|_{X_t}$ of $E$ on each fiber $X_t$ has simple normal crossings. Let $\Omega\subset \cX$ be a coordinate subset on $\cX$. We take $(z_1,\dots z_n, t _1,\dots, t_m)$ a coordinate system on $\Omega$ such that the last $m$ variables $t_1,\dots, t_m$ corresponds to the map $p$ itself and such that $z_1\dots z_r= 0$ is the local equation of $E\cap \Omega$.
\begin{enumerate}
\item The metric $h_L$ has \textbf{Poincar\'e type singularities} along $E$, i.e., its local weights $\psi_L$ on $\Omega$ can be written as
\begin{equation*}
\psi_L \equiv - \sum_{I} b_I\log\left(\big(\prod_{i\in I}|z_i|^{2m_i}\big)\big(\phi_I(z)-\log \big(\prod_{i\in I}|z_i|^{2k_i}\big)\big)\right)
\end{equation*}
modulo $\mathcal C^\infty$ functions, where $b_I$ are positive real numbers for all $I$, $m_i, k_i$ are positive integers. All $(\phi_I)_I$ are smooth functions on $\Omega$. The set of indexes in the sum coincides with the non-empty subsets of $\{1,\dots, r\}$.
\item  The metric $h_L$ has \textbf{logarithmic type singularities} along $E$, i.e., its local weights $\psi_L$ on $\Omega$ can be written as
\begin{equation*}
\psi_L \equiv - \sum_{I} b_I\log \big(\phi_I(z)-\log (\prod_{i\in I}|z_i|^{2k_i}) \big)
\end{equation*}
modulo $\mathcal C^\infty$ functions, where $b_I$ are positive real numbers satisfying that $b_I < 1$ for all $I$, all $k_i$ are positive integers and $(\phi_I)_I$ are smooth functions on $\Omega$. The set of indexes in the sum coincides with the non-empty subsets of $\{1,\dots, r\}$.
\item The metric $h_L$ has \textbf{Kawamata (klt) type singularities} along $E$, i.e., its local weights $\psi_L$ on $\Omega$ can be written as
\begin{equation*}
\psi_L \equiv  \sum_{i\in I} a_i\log |z_i|^2
\end{equation*}
modulo $\mathcal C^\infty$ functions, where $a_i$ are real numbers satisfying that $a_i < 1$ for all $i$. The set of indexes in the sum coincides with the non-empty subsets of $\{1,\dots, r\}$.
\end{enumerate}
\end{ex}

In Section \ref{section-3}, we aim to remove assumption \ref{Z4} by investigating the case where the metric $h_L$ has Kawamata-type singularities. We utilize the strong decomposition in $L^2$ Hodge theory. In other words, we attempt to generalize Berndtsson's seminal result in \cite{Bo09} to this singular setting.

\subsection*{Acknowledgement}
The author would like to express his gratitude to Prof. Shigeharu Takayama for his guidance and warm encouragement. The author is also grateful to Prof. Junyan Cao for his comments on this manuscript.

\section{Nakano positivity via the optimal $L^2$-estimate} \label{section-2}
In this section, we will prove Theorem \ref{Thm2}. In \cite{DNWZ20}, the authors investigate the positivity properties of Hermitian holomorphic vector bundles using $L^p$-estimates of the $\bar\partial$ operator. They introduce several $L^p$-estimate (extension) conditions, among which one is referred to as the optimal $L^2$-estimate condition.

\begin{defn}\label{def:Lp estimate}
Let $(X,\omega)$ be a K\"{a}hler manifold of dimension $n$, which admits a positive Hermitian holomorphic line bundle and $(E,h)$ be a (singular) Hermitian vector bundle over $X$. The vector bundle $(E,h)$ satisfies \emph{the optimal $L^2$-estimate condition}
if for any positive Hermitian holomorphic line bundle $(A,h_A)$ on $X$,
for any $f\in\mathcal{C}^\infty_c(X,\wedge^{n,1}T^*_X\otimes E\otimes A)$ with $\bar\partial f=0$,
there is $u\in L^p(X,\wedge^{n,0}T_X^*\otimes E\otimes A)$, satisfying $\bar\partial u=f$ and
$$\int_X|u|^2_{h\otimes h_A}dV_\omega\leq \int_X\langle B_{A,h_A}^{-1}f,f\rangle dV_\omega,$$
provided that the right-hand side is finite,
where $B_{A,h_A}=[i\Theta_{A,h_A}\otimes \text{Id}_E,\Lambda_\omega]$.
\end{defn}

One of the main results in \cite{DNWZ20} was the following characterization of Nakano positivity in terms of optimal $L^2$-estimate condition.

\begin{thm} \cite[Theorem 1.1]{DNWZ20} \label{thm:theta-nakano text_intr}
Let $(X,\omega)$ be a  K\"{a}hler manifold of dimension $n$ with a K\" ahler metric $\omega$, which admits a positive Hermitian holomorphic  line bundle, $(E,h)$ be a smooth Hermitian vector bundle over $X$,
and $\theta\in \mathcal{C}^0(X,\wedge^{1,1}T^*_X\otimes End(E))$ such that $\theta^*=\theta$.
If for any $f\in\mathcal{C}^\infty_c(X,\wedge^{n,1}T^*_X\otimes E\otimes A)$ with $\bar\partial f=0$,
and any positive Hermitian line bundle $(A,h_A)$ on $X$ with $i\Theta_{A,h_A}\otimes Id_E+\theta>0$ on $\text{supp}f$,
there is $u\in L^2(X,\wedge^{n,0}T_X^*\otimes E\otimes A)$, satisfying $\bar\partial u=f$ and
$$\int_X|u|^2_{h\otimes h_A}dV_\omega\leq \int_X\langle B_{h_A,\theta}^{-1}f,f\rangle_{h\otimes h_A} dV_\omega,$$
provided that the right-hand side is finite,
where $B_{h_A,\theta}=[i\Theta_{A,h_A}\otimes \text{Id}_E+\theta,\Lambda_\omega]$,
then $i\Theta_{E,h}\geq\theta$ in the sense of Nakano.
On the other hand, if in addition $X$ is assumed to have a complete K\"ahler metric,
the above condition is also necessary for that $i\Theta_{E,h}\geq\theta$ in the sense of Nakano.
In particular, if $(E,h)$ satisfies the optimal $L^2$-estimate condition, then $(E,h)$ is Nakano semi-positive.
\end{thm}

\begin{rem}\label{rem:reduce to trivial bundle}
As remark 1.2 in \cite{DNWZ20} said, if $X$ admits a strictly plurisubharmonic function, we can take $A$ to be the trivial bundle (with nontrivial metrics).
\end{rem}

We aim to use this theorem to study Nakano positivity in the singular setting. In general, the direct image sheaf was not the smooth vector bundle and therefore we can not use Theorem \ref{thm:theta-nakano text_intr}, but if the singular metric $h_L$ of the line bundle $L$ satisfies the Assumption \ref{Z1}, then the direct image sheaf has a smooth $L^2$ canonical metric. The following lemmas are crucial for solving the $\dbar$-equation with $L^2$ estimate, which is essential for establishing the desired positivity results in this setting.

\begin{lem} \cite[Lemma 3.2]{Dem82} \cite[Appendix]{DNWZ20} \label{operator}
Let $X$ be a complex manifold with dimension $n$, assume that $\theta \in \wedge^{1,1}T_X^*$ be a positive $(1,1)$-form, and fix an integer $q\geq 1$.
\begin{enumerate}
\item  for each form $u \in \wedge^{n,q}T_X^*, \langle [\theta, \Lambda_{\omega}]^{-1} u, u \rangle dV_{\omega}$ is non-increasing with respect to $\theta$ and $\omega$;
\item  for each form $u \in \wedge^{n,1}T_X^*, \langle [\theta, \Lambda_{\omega}]^{-1} u, u \rangle dV_{\omega}$ is independent with respect to $\omega$.
\end{enumerate}
\end{lem}

We need the Richberg-type global regularization result for unbounded quasi-plurisubharmonic functions. Recall an upper semi-continuous function $\phi: X \rightarrow [-\infty, +\infty)$ on a complex manifold $X$ is quasi-psh if it is locally of the form $\phi = u + f$ where $u$ is plurisubharmonic(psh) function and $f$ is a smooth function.

\begin{lem} \cite[Theorem 3.8]{Bou17} \label{rich}
Let $\phi$ be a quasi-psh function on a complex $X$, and assume given finitely many closed, real $(1,1)$-forms $\theta_{\alpha}$ such that $\theta_{\alpha} + i \ddbar \phi \geq 0$ for all $\alpha$. Suppose either that $X$ is strongly pseudoconvex, or that $\theta_{\alpha}>0$ for all $\alpha$. Then we can find a sequence $\phi_j \in \mathcal{C}^{\infty}(X)$ with the following properties:
\begin{enumerate}
\item  $\phi_j$ converges point-wise to $\phi$;
\item  for each relatively compact open subset $U \Subset X$, there exists $j_U\gg 1$ such that the sequence $(\phi_j)$ becomes decreasing with $\theta_{\alpha} + i \ddbar \phi_j >0$ for each $\alpha$ when $j\geq j_U$.
\end{enumerate}
\end{lem}

The following $L^2$-estimate for the $\dbar$ equation is fundamental in complex geometry.

\begin{lem}\cite[Theorem 4.5]{bookJP} \label{thm: L2 estimate Nakano}
Let $(X,\omega)$ be a complete K\"ahler manifold, with a K\"ahler metric which is not necessarily complete.  Let $(E,h)$ be a Hermitian  vector bundle of rank $r$ over $X$, and assume that the curvature operator $B:=[i\Theta_{E,h},\Lambda_\omega]$ is semi-positive definite everywhere on $\wedge^{n,q}T_X^*\otimes E$, for some $q\geq 1$. Then for any form $g\in L^2(X,\wedge^{n,q}T^*_{X}\otimes E)$ satisfying $\bar{\partial}g=0$ and $\int_X\langle B^{-1}g,g\rangle dV_\omega<+\infty$, there exists $f\in L^2(X,\wedge^{n,q-1}T^*_X\otimes E)$ such that $\bar{\partial}f=g$ and $$\int_X|f|^2dV_\omega\leq \int_X\langle B^{-1}g,g\rangle dV_\omega.$$
\end{lem}

\begin{thm} \label{equ}
Let $p:\cX\to D$ be a holomorphic proper fibration from a $(n+m)$-dimensional K\"ahler manifold $\cX$ onto the bounded pseudoconvex domain $D\subset \mathbb C^{m}$, and let $(L,h_L)$ be a holomorphic line bundle endowed with a possibly singular hermitian metric $h_L$ with local weight $\psi$ and curvature current $ i \Theta_{h_L}(L)\ge 0$. We assume that the K\"ahler manifold $\cX$ contains a Stein Zariski open subset and $\phi$ be any smooth strictly plurisubharmonic function on $D$. If
$$ v \in L^2_{loc}(\cX, \wedge^{n,1} T^{\ast}_{\cX} \otimes L)
$$
satisfying $\dbar v=0$ and
$$ \int_{\cX} \langle [i \ddbar p^{\ast}\phi, \Lambda_{\omega}]^{-1} v, v \rangle_{\psi} e^{-p^{\ast}\phi} dV_{\omega} < \infty.
$$
Then $v= \dbar u$ for some $u\in L^2(\cX, \wedge^{n,0} T^{\ast}_{\cX} \otimes L)$ such that
\begin{equation} \label{eq-optimal}
\int_{\cX} |u|^2_{\psi} e^{-p^{\ast}\phi} dV_{\omega} \leq \int_{\cX} \langle [i \ddbar p^{\ast}\phi, \Lambda_{\omega}]^{-1} v, v \rangle_{\psi} e^{-p^{\ast}\phi} dV_{\omega}.
\end{equation}
Here the subscript $|\cdot|^2_{\psi}, \langle \cdot \rangle_{\psi}$ means the inner product with respect to metric weight $\psi$ of $L$.
\end{thm}

\begin{proof}
Firstly, we note $h_L e^{-p^{\ast}\phi}$ is also the singular metric of $L$ because $p^{\ast}\phi$ be a globally function on $\cX$. Therefore we have $i \Theta_{h_L}(L) + i \ddbar p^{\ast}\phi \geq i \ddbar p^{\ast}\phi$ in the sense of currents.
To prove the claim, we need an $L^2$-version of the Riemann extension principal. This is to say, if $\alpha\in L^2_{loc}$ be a $L$-valued form on a complex $X$ such that $\dbar \alpha=\beta$ outside a closed analytic subset $A \subset X$, then $\dbar \alpha = \beta$ holds on the whole $X$. On the other hand, if $X$ is a Stein manifold and $L$ be a line bundle on $X$, there exists a hypersurface $H \subset X$ such that $X \setminus H$ is Stein and $L$ is trivial on $X\setminus H$.
Thanks to this, we can assume that $\cX$ is Stein and $L$ is trivial on $\cX$. Then the metric $h_L= e^{-\psi}$ and its local weight $\psi$ is globally defined on $\cX$. Now we can use the global regularization of unbounded quasi-psh functions.

By the Lemma \ref{rich}, we may find an exhaustion of $\cX$ by weakly pseudoconvex open subsets $\Omega_j$ such that $\psi_j = \psi|_{\Omega_j}$ is the decreasing limit of sequence $\psi_{j,k} \in \mathcal C^{\infty}(\Omega_j)$ with
$$ i \ddbar \psi_{j,k} \geq 0  \Longrightarrow  i \ddbar \psi_{j,k} + i \ddbar p^{\ast}\phi \geq i \ddbar p^{\ast}\phi.
$$
Because weakly pseudoconvex manifold admits a complete K\"ahler metric, on $\Omega_j$ we can solve the classical $\dbar$ equation with the $L^2$-estimate as Lemma \ref{thm: L2 estimate Nakano}, i.e., there exist $u_{j,k} \in L^2(\Omega_j, \wedge^{n,0} T^{\ast}_{\Omega_j} \otimes L)$ such that $\dbar u_{j,k} = v$ on $\Omega_j$ and
\begin{align}
\nonumber \int_{\Omega_j} |u_{j,k}|^2_{\psi_{j,k}} e^{-p^{\ast}\phi} dV_{\omega} &=  \int_{\Omega_j} |u_{j,k}|^2 e^{-\psi_{j,k}} e^{-p^{\ast}\phi} dV_{\omega} \\
\nonumber &\leq \int_{\Omega_j} \langle [i \ddbar p^{\ast}\phi, \Lambda_{\omega}]^{-1} v, v \rangle_{\psi_{j,k}} e^{-p^{\ast}\phi} dV_{\omega} \\
\nonumber&\leq \int_{\Omega_j} \langle [i \ddbar p^{\ast}\phi, \Lambda_{\omega}]^{-1} v, v \rangle_{\psi_{j}} e^{-p^{\ast}\phi} dV_{\omega} \\
\nonumber&\leq \int_{\cX} \langle [i \ddbar p^{\ast}\phi, \Lambda_{\omega}]^{-1} v, v \rangle_{\psi} e^{-p^{\ast}\phi} dV_{\omega} \\
\nonumber&= M (constant).
\end{align}
The second inequality because of $\psi_j$ is the decreasing limit of a sequence $\psi_{j,k}$. By monotonicity of $(\psi_{j,k})_k$, we know the integration $\int_{\Omega_j} |u_{j,k}|^2 e^{-\psi_{j,l}} e^{-p^{\ast}\phi} dV_{\omega} \leq M$ for $k\geq l$, this shows in particular that $(u_{j,k})_k$ is bounded in $L^2(\Omega_j, e^{-p^{\ast}\phi-\psi_{j,l}})$. After passing to the subsequence, we thus assume that $u_{j,k}$ converges weakly in $L^2(\Omega_j, e^{-p^{\ast}\phi-\psi_{j,l}})$ to $u_j$, which may further be assumed to be the same for all $l$, by a diagonal argument. Now we have $\dbar u_j = v$, and $\int_{\Omega_j} |u_{j}|^2 e^{-\psi_{j,l}} e^{-p^{\ast}\phi} dV_{\omega} \leq M$ for all $l$, therefore $\int_{\Omega_j} |u_{j}|^2 e^{-\psi} e^{-p^{\ast}\phi} dV_{\omega} \leq M$ by monotone convergence of $\psi_{j,l} \rightarrow \psi$. Once again by a diagonal argument, we may arrange that $u_j \rightarrow u$ weakly in $L^2(K, e^{-\psi})$ for each compact subset $K \subset X$, and finally we are led to the desired conclusion.
\end{proof}

We can now prove Theorem \ref{Thm2} by following the approach of Deng, Ning, Wang, and Zhou.
\begin{thm}\label{thm: direct image-optimal L2 estimate}
Under the set-up of case $Z$, the Hermitian holomorphic vector bundle $(\mathcal F, \|\cdot\|)$ over $D$ defined in \eqref{F}
satisfies the optimal $L^2$-estimate condition.
\end{thm}

\begin{proof}
According to Theorem \ref{thm:theta-nakano text_intr}, it suffices to prove that $(\mathcal F = p_*(K_{\cX/D} + L), \|\cdot\|)$
satisfies the optimal $L^2$-estimate condition with the standard K\"ahler metric $\omega_0$ on $D \subset\mc^n$. Let $\omega$ be an arbitrary K\"ahler metric on $\cX$.

Let $f$ be a $\bar\partial$-closed compact supported smooth $(m,1)$-form with values in $\mathcal F$, and let $\phi$ be any smooth strictly plurisubharmonic function on $D$. We can write $f(t)=dt\wedge(f_1(t)d\bar t_1+\cdots +f_n(t)d\bar t_n)$, with $f_i(t)\in \mathcal F_t=H^0(X_t, K_{X_t}\otimes L)$. One can identify  $f$ as a smooth compact supported $(n+m,1)$-form $\tilde f(t,z):=dt\wedge (f_1(t,z)d\bar t_1+\cdots+f_n(t,z)d\bar t_n)$ on $\cX$,
with $f_i(t,z)$ being holomorphic section of $K_{X_t}\otimes L|_{X_t}$.
We have two observations as follows, the first is that $\dbar_z f_i(t,z)=0$ for any fixed $t\in D$, since  $f_i(t,z)$ are holomorphic sections $K_{X_t}\otimes L|_{X_t}$. The second is that $\dbar_tf=0$, since $f$ is a $\bar\partial$-closed form on $D$.
It follows that $\tilde f$ is a $\bar\partial$-closed compact supported  $(n+m,1)$-form on $\cX$ with values in $L$.
We want to solve the equation $\bar\partial\tilde u=\tilde f$ on $X$ by using Theorem \ref{equ}.
Now we equipped $L$ with the metric $\tilde h:=he^{-\pi^*\phi}$,
then $i\Theta_{L,\tilde h}=i\Theta_{L,h}+i\ddbar\pi^*\phi$, which is also  semi-positive in the sense of currents. Hence there is  $\tilde u\in \wedge^{m+n,0}T^{\ast}_{\cX}\otimes L$, such that $\bar\partial\tilde u=\tilde f$, and satisfies the following estimate
\begin{align}\label{eqn: optimal L2 estimate 1}
\int_{\cX} c_{m+n}\tilde u\wedge \bar{\tilde u}e^{-p^*\phi} =& \int_{\cX} |\tilde u|^2 e^{-p^*\phi}dV_{\omega}  \notag\\
\leq &\int_{\cX}\langle [i\partial\bar\partial p^*\phi, \Lambda_{\omega}]^{-1}\tilde f,\tilde f\rangle e^{-p^*\phi}dV_\omega \notag\\
=& \int_{\cX}\langle [i\partial\bar\partial p^*\phi, \Lambda_{\omega'}]^{-1}\tilde f,\tilde f\rangle e^{-p^*\phi}dV_{\omega'} \notag\\
=&\int_D \langle [i\partial\bar\partial \phi, \Lambda_{\omega_0}]^{-1}f,f\rangle_te^{-\psi}dV_{\omega_0}.
\end{align}
The first inequality due to \eqref{eq-optimal}, the second equality holds because $\tilde f$ is $(n+m,1)$-form, and therefore $\langle [i\partial\bar\partial p^*\phi, \Lambda_{\omega}]^{-1}\tilde f,\tilde f\rangle dV_\omega$ are independent to $\omega$ in view of Lemma \ref{operator}. The last equality is valid because here we choose $\omega'= i \Sigma_{j=1}^m dt_j \wedge d\bar t_j + i \Sigma_{j=1}^n dz_j \wedge d\bar z_j$. The notation $\langle\cdot,\cdot\rangle_t$ here means a pointwise inner product concerning the Hermitian metric $\|\cdot\|$ on $\mathcal F$.

Set $\tilde u_t:=\tilde u(t,\cdot)$, we observe that $\dbar \tilde u_t=0$ for any fixed $t\in D$, since $\bar\partial\tilde u=\tilde f$ and the $(n+m,1)$-form $\tilde f$ contains only the terms of $d\bar t_i$.
This means that $\tilde u_t \in \mathcal F_t$, and hence we may view $\tilde u$ as a section $u$ of $\mathcal F$.
It is obviously that $\bar\partial u=f$.
Due to Fubini's theorem, we have that
\begin{align}\label{eqn: optimal L2 estimate 2}
\int_{\cX} c_{m+n}\tilde u\wedge \bar{\tilde u}e^{-p^*\phi} = \int_D \|u_t\|^2e^{-\phi}dV_{\omega_0}.
\end{align}
Combining \eqref{eqn: optimal L2 estimate 1} with \eqref{eqn: optimal L2 estimate 2}, we obtain
\begin{align*}
\int_D \|u_t\|^2e^{-\phi} dV_{\omega_0} \leq \int_D \langle [i\partial\bar\partial \phi, \Lambda_{\omega_0}]^{-1}f,f\rangle_t e^{-\phi}dV_{\omega_0}.
\end{align*}
So we have proved that $\mathcal F$ satisfies the optimal $L^2$-estimate condition, thus owing to Theorem \ref{thm:theta-nakano text_intr} (and Remark \ref{rem:reduce to trivial bundle}),  $\mathcal F$ is semi-positive in the sense of Nakano.
\end{proof}

\section{Nakano positivity via curvature computation} \label{section-3}

In this section, we will investigate the Nakano positivity of direct image sheaves under the K\"ahler condition alone, thereby eliminating the need for Assumption \ref{Z4}. We adopt Berndtsson's approach from \cite{Bo09}.

\subsection{Set-up of case $V$} \label{setup}

Let $p:\cX\to D$ be a holomorphic proper fibration from a $(n+m)$-dimensional K\"ahler manifold $\cX$ onto the polydisk $D\subset \mathbb C^{m}$, and let $(L,h_L)$ be a holomorphic line bundle endowed with a possibly singular hermitian metric $h_L$. We assume that there exists a divisor $E= E_1+\dots + E_N$  whose support is contained in the total space $\cX$ of $p$
such that the following requirements are fulfilled.
\begin{enumerate} [label={\color{violet}(V.\arabic*)}]

\item \label{V1} The divisor $E$ intersect each fiber transversally, i.e., for every $t\in D$ the restriction divisor $E_t:= E|_{X_t}$ of $E$ on each fiber $X_t$ has simple normal crossings. Let $\Omega\subset \cX$ be a coordinate subset on $\cX$. We take $(z_1,\dots z_n, t _1,\dots, t_m)$ a coordinate system on $\Omega$ such that the last $m$ variables $t_1,\dots, t_m$ corresponds to the map $p$ itself and such that $z_1\dots z_r= 0$ is the local equation of $E\cap \Omega$.

\item \label{V2} The metric $h_L$ has \textbf{klt type singularities} along $E$, i.e., its local weights $\psi_L$ on $\Omega$ can be written as
\begin{equation*}
\psi_L \equiv  \sum_{i\in I} a_i\log |z_i|^2
\end{equation*}
modulo $\mathcal C^\infty$ functions, where $a_i$ are real numbers satisfying that $a_i < 1$ for all $i$. The set of indexes in the sum coincides with the non-empty subsets of $\{1,\dots, r\}$.

\item \label{V3} The Chern curvature of $(L,h_L)$ satisfies
\[i\Theta_{h_L}(L)\ge 0\]
in the sense of currents on $\cX$.
\end{enumerate}

\medskip

\noindent One wants to study some interesting objects such as adjoint bundles $K_{\cX/D}+L$
or $(K_{\cX/D}+L)\otimes \mathcal I(h_L)$ and their direct image sheaves.
Under the setting of case $V$, we set
\begin{align} \label{G}
\mathcal G:=p_*((K_{\cX/D}+L)\otimes \mathcal I(h_L)) = p_*(K_{\cX/D}+L),
\end{align}
the last equality holds because $\mathcal I(h_L) = \mathcal{O}_{\cX}$ which is due to the conditions in Assumption \ref{V2}. We first remark that $\cG$ is indeed a vector bundle by the K\"ahler version of Ohsawa--Takegoshi theorem \cite{CaoOT} under the assumption \ref{V3}, which implies that any element of $H^0\left(X_t, K_{X_t}+ L\right)$ extends to $\cX$. One obtains
\begin{equation*}
\cG_t= H^0\left(X_t, K_{X_t}+ L \right)
\end{equation*}
for every $t\in D$. Therefore there exists an induced canonical $L^2$ metric on $\cG$. We will interchangeably denote by $\|\cdot \|$ or $h_\cG$ the $L^2$ metric on $\mathcal G$. If $u$ is the section of $\mathcal G$, $u_t:= u|_{X_t} \in \mathcal G_t=H^0(X_t, K_{X_t}+L)$, then
$$\|u_t\|^2:=c_n \int_{X_t}u\wedge \bar u e^{-\psi_{L}} ~~\text{with} ~~ c_n=(\sqrt{-1})^{n^2}.$$
By assumption, we claim that the $L^2$ metric is smooth on $D$. Indeed, we can use partitions of unity to reduce to checking that integrals of the form $\int_{\Omega\cap X_t} |u_t|^2e^{-\psi_L}$ vary smoothly with $t$, where $\psi_L$ is given by the expression in  \ref{V2}. Now it is clear there that all derivatives in the $t,\bar t$ variables of $\psi_L$ are bounded and smooth so that the result follows from general smoothness results for integrals depending on a parameter. We denote by $\nabla$ the Chern connection of $(\mathcal G, \|\cdot \|)$ on $D$, and $\nabla^{1,0}, \nabla^{0,1}$ represent the $(1,0)$ part and $(0,1)$ part of the Chern connection respectively.

We set $\cX^\circ:=\cX\setminus E$, $X_t^\circ:=X_t \cap \cX^\circ$ be a open fiber, $L_t:=L|_{X_t}$, $h_{L_t}:=h_{L}|_{X_t}$. When working on a trivializing coordinate chart of $L$, we will always denote by $\psi_L$ the local weight of $h_L$. Under assumption \ref{V1}, we will write  $E :=\sum_{i=1}^N  E_i$ for the decomposition of $E$ into its (smooth) irreducible components.  Next, let $s_i$ be a section of $\mathcal O_\cX(E_i)$ that cuts out $E_i$, and let $h_i$ be a smooth hermitian metric on $\mathcal O_\cX(E_i)$.
In the following, $|s_i|^2$ stands for $|s_i|^2_{h_i}$, and we assume that $|s_i|^2<e^{-1}$. Let $\omega$ be a fixed K\"ahler metric on $\cX$, and let
\begin{equation} \label{pmetric}
\omega_E  := C\omega + dd^c \Big[ - \sum_{i=1}^N \log \log \frac{1}{|s_i|^2} \Big]  \qquad\text{on } \cX^\circ
\end{equation}
be a metric with Poincar\'e singularities along $E$, the constant $C$ here is large enough to make $\omega_E$ positive. Thanks to \ref{V1} we infer that $\displaystyle \omega_E |_{X_t^\circ}$ is a complete K\"ahler metric on $X^\circ_t$ with Poincar\'e singularities along $E\cap X_t$ for each $t\in D$.

On each compact K\"ahler fiber with dimension $n$, the section $u$ is a smooth $(n,0)$-form, we have the next equality which is very important for our case,

 \begin{equation} \label{integ}
 \|u_t\|^2= c_n \int_{X_t}u\wedge \bar{u} e^{-\psi_L} = c_n \int_{X_t^\circ}u\wedge \bar{u} e^{-\psi_L} = c_n \int_{X_t^\circ} |u|_{\omega_E}^2 e^{-\psi_L} dV_{\omega_E}.
\end{equation}

\noindent Here $\omega_E$ is the metric in \eqref{pmetric}. The second equality since the divisor has measure zero, and the third equality holds thanks to the integrand is independent to $\omega_E$. Due to \eqref{integ}, we can use the $L^2$-Hodge theory on $X_t^{\circ}$ with respect to metric $\omega_E$ and $\psi_L$.

\begin{thm} \label{Thm1}
Under the set-up of case $V$, the Hermitian holomorphic vector bundle $(\mathcal G, \|\cdot\|)$ over $D$ defined in \eqref{G} is semi-positive in the sense of Nakano.
\end{thm}

\subsection{the Chern connection}
As in the set-up of case $V$, Let $u$ be the section of direct image sheaf $\mathcal{G}$, it is an $L$-valued $(n,0)$ form on $\cX$. By the canonical $L^2$ metric we have $\|u_t\|^2= c_n \int_{X_t}u\wedge \bar{u} e^{-\psi_L}$.

One well-known fact is that for any complex manifold $X$ with dimension $n$, and two positive $(1,1)$-forms $\omega$ and $\widetilde{\omega}$ on $X$ with relation $\omega \leq \widetilde{\omega}$. Let $\alpha$ be a $(n,q)$-form on $X$, then $ |\alpha|_{\omega}^2 dV_{\omega} \geq |\alpha|_{\widetilde{\omega}}^2 dV_{\widetilde{\omega}}$. In particular, if $\alpha$ is a $(n,0)$-form, then we have $ |\alpha|_{\omega}^2 dV_{\omega} = |\alpha|_{\widetilde{\omega}}^2 dV_{\widetilde{\omega}}$.

Therefore we have the next equality which is very important for our case,
\begin{equation}
 \|u_t\|^2= c_n \int_{X_t}u\wedge \bar{u} e^{-\psi_L} = c_n \int_{X_t^\circ}u\wedge \bar{u} e^{-\psi_L} = c_n \int_{X_t^\circ} |u|_{\omega_E}^2 e^{-\psi_L} dV_{\omega_E}.
\end{equation}
\noindent Here $\omega_E$ is the metric with Poincar\'e singularities along $E$ as in \eqref{pmetric}. Due to this, we can use the Hodge decomposition on $X_t^{\circ}$ with respect to metric $\omega_E$ and $\psi_L$. For any two section $u,v$ of $\mathcal G$, put $[u, v] = c_n u\wedge \bar{v} e^{-\psi_L}$, then $ \|u_t\|^2 = \int_{X_t}[u, u] = p_{\ast}[u_t,u_t]$, and consequently we define the inner product on $\mathcal G$ by
$$(u,v)= p_{\ast}[u,v]= p_{\ast}(c_n u\wedge \bar{u} e^{-\psi_L}). $$
Let $\cX$ be a K\"ahler manifold of dimension $m+n$ which is smoothly fibered over connected complex $m$-dimensional polydisk $D$, the local coordinates on $\cX$ are denote by $(z_1, \cdot\cdot\cdot, z_n, t_1, \cdot \cdot\cdot, t_m)$. We define a complex structure on $\mathcal G$ by saying the smooth section $u$ defines a holomorphic section of $\mathcal G$ if $u \wedge dt$ is a holomorphic local section of $K_{\cX}\otimes L$. Let $u$ be a smooth local section of $\mathcal{G}$, $u_t \in H^0(X_t, K_{X_t}+L_t)$,
this is an $L$-valued $(n,0)$-form on $X_t$. For example, if locally we write $u=fdz + \sum g_i\h{dz_i}\wedge dt_i$, then both $u$ and $fdz$ represent the same section of $\mathcal{G}$.
It is obvious that $\dbar u$ also be a smooth form on $\cX$. On account of $\dbar u \wedge dt \wedge d\bar{t} =0$, thus
$$ \dbar u= \sum \nu^j \wedge d \bar t_j + \sum \eta^j \wedge dt_j,
$$
where the $\nu^j$ and $\eta^j$ are smooth forms on $\cX$ and $\nu^j$ defines a section to $\mathcal{G}$. We define the $(0,1)$-part of the connection $\nabla$ on $\mathcal{G}$ by letting
\begin{equation}
\nabla^{0,1}u = \sum \nu^j \wedge d \bar t_j.
\end{equation}

\noindent Therefore $u$ is a holomorphic section if and only if $\dbar u \wedge dt = 0$, or if and only if $\dbar u = \sum \eta^j \wedge d t_j$.

Now we define the $(1,0)$-part of the Chern connection. If $u$ is a smooth section of $\mathcal{G}$, let $u^{\circ}:= u|_{\cX^{\circ}}$, We eliminate all terms containing $dt_j$ in $u^{\circ}$, as they represent the same section of $\mathcal{G}$.
Then we have
\begin{equation} \label{mu}
\partial^{\psi_L} u^{\circ} := e^{\psi_L} \partial(e^{-\psi_L} u^{\circ}) = (-\partial \psi_L) u^{\circ} + \partial u^{\circ} = \sum \mu^j \wedge d t_j.
\end{equation}
for some $L$-valued $(n,0)$-forms $\mu^j$ that contains terms of $\partial \psi_L$ on $\cX^{\circ}$. As we have seen in assumption \ref{V2}, the weight $\psi_L$ has singularities along the divisor $E$, so $\psi_L$ is smooth outside $E$. And we have $\mu^j$ is $L^2$ integrable with respect to $(\omega_E, e^{-\psi_L})$, since here $\partial \psi_L$ is taking the derivative with respect to the $t_j$ variables, not the $z_i$ variables (We choose $u^{\circ}$ such that it does not contain any $dt_j$ terms.).


The restrictions of $\mu^j$ defined in \eqref{mu} on each open fiber $X_t^{\circ}$ are in general not holomorphic. So we let $P(\mu^j)$ be the orthogonal projection of $\mu^j$ onto the space of holomorphic forms on each open fiber. Similar to the Remark $3.5$ in \cite{CGP}, since $P(\mu^j)$ is holomorphic on $X_t^{\circ}$ and $L^2$-integrable, therefore it extends holomorphically to the compact fiber $X_t$. We denote by $P'(\mu^j)$ the extension of $P(\mu^j)$. Similar to \cite[Lemmma 4.1]{Bo09}, we define the $(1,0)$-part of the Chern connection by
\begin{equation}
 \nabla^{1,0} u = \sum P'(\mu^j)d t_j.
\end{equation}

\noindent We also need to verify that
$$ \partial_{t_j}(u, v) = (P'(\mu^j), v) + (u, \dbar_{t_j}v)
$$
for any smooth section $u$ and $v$ to $\mathcal{G}$, here $\dbar v = \sum v^j\wedge d \bar t_j + \sum w^j \wedge d t_j$. We have defined that $\nabla^{0,1} v = \sum v^j \wedge d \bar t_j$, sometimes we write $v^j = \dbar_{t_j}v$. On the other hand, by the commutativity of $p_{\ast}$ and $\partial$, we have $\partial (u, v) =  \partial p_{\ast}([u, v]) = \partial \int_{X_t} c_n u \wedge \bar{v} e^{-\psi_L} = \partial \int_{X^{\circ}_t} c_n u \wedge \bar{v} e^{-\psi_L}$. Hence, one obtains
\begin{align*}
 \partial (u, v) = & \int_{X^{\circ}_t} c_n \partial^{\psi_L} u \wedge \bar{v} e^{-\psi_L} + (-1)^n \int_{X^{\circ}_t} c_n u \wedge \overline{\dbar v} e^{-\psi_L}\\
= & \int_{X^{\circ}_t} c_n \sum \mu^j \wedge d t_j \wedge \bar{v} e^{-\psi_L} + (-1)^n \int_{X^{\circ}_t} c_n u \wedge \sum\bar{v}^j\wedge d t_j e^{-\psi_L} \\
&+ (-1)^n \int_{X^{\circ}_t} c_n u \wedge \sum\bar{w}^j\wedge d \bar t_j e^{-\psi_L}.
\end{align*}
The last term above vanishes owing to the degree of forms, thus we obtain
\begin{align*}
\partial (u, v) =& \int_{X_t} c_n \sum \mu^j \wedge d t_j \wedge \bar{v} e^{-\psi_L} + (-1)^n \int_{X_t} c_n u \wedge \sum\bar{v}^j\wedge d t_j e^{-\psi_L} \\
=& \sum((\mu^j, v) + (u, v^j))d t_j \\
=& \sum ((P'(\mu^j), v) + (u, \dbar_{t_j}v))d t_j,
\end{align*}
and so we are led to the conclusion that the connection is compatible with the inner product.

\subsection{Nakano positivity}
To study the Nakano positivity, we will use the so-called $\ddbar$-Bochner-Kodaira technique illustrated in section $2$ in \cite{Bo09}. Let $E$ be a holomorphic vector bundle with connection $D$. Let $u_j$ be an $m$-tuple of holomorphic sections to $E$, satisfying $D u_j=0$ at fixed point $0$. Put $T_u = \sum (u_j, u_k) \widehat{dt_j \wedge d \bar{t}_k}$, here $\widehat{dt_j \wedge d \bar{t}_k}$ denotes the product of all differential $d t_i$ and $d \bar{t}_i$, except $dt_j$, $d \bar{t}_k$ and multiplied by a number of modulus $1$, so that $T_u$ is non-negative. Then by calculation we obtain $ i \ddbar T_u = - \sum(\Theta^E_{j,k}u_j, u_k) dV_t$. So $E$ is Nakano positivity at the given point if and only if this expression $i \ddbar T_u$ is negative for any choice of holomorphic sections $u_j$ satisfying $D u_j =0$ at the chosen point.

So we let $u_j$ be an $m$-tuple of holomorphic sections of $\mathcal{G}$ satisfying $\nabla^{1,0} u_j = 0$ at a given point $t=0$. Let
$$ T_u = \sum (u_j, u_k) \widehat{dt_j \wedge d \bar{t}_k}
$$
as above so that $T_u$ is nonnegative.
We put $\h{u} = \sum u_j \wedge \h{dt_j}$ be an $L$-valued $(N,0)$-form with $N = n+m-1$, and thus
$$ T_u = c_N p_{\ast}(\h{u} \wedge \ol{\h{u}} e^{-\psi_L}).
$$
Then we obtain
 \begin{align} \label{dbart}
 \dbar T_u &= c_N p_{\ast}(\dbar \h{u} \wedge \ol{\h{u}} e^{-\psi_L}) + (-1)^N c_N p_{\ast}(\h{u} \wedge \ol{\partial^{\psi_L}\h{u}} e^{-\psi_L})\\
 \nonumber &= (-1)^N c_N p_{\ast}(\h{u} \wedge \ol{\partial^{\psi_L}\h{u}} e^{-\psi_L}).
\end{align}
Therefore
$$ \partial \dbar T_u = (-1)^N c_N p_{\ast}(\partial^{\psi_L} \h{u} \wedge \ol{\partial^{\psi_L}\h{u}} e^{-\psi_L}) + c_N p_{\ast}(\h{u} \wedge \ol{\dbar\partial^{\psi_L}\h{u}} e^{-\psi_L}).
$$
Using the identity $\dbar \partial^{\psi_L} + \partial^{\psi_L} \dbar = \partial \dbar \psi_L$, we can change it to the next equation,
\begin{align} \label{ddbart}
\partial \dbar T_u = (-1)^N c_N p_{\ast}(\partial^{\psi_L} \h{u} \wedge \ol{\partial^{\psi_L}\h{u}} e^{-\psi_L}) - c_N p_{\ast}(\h{u} \wedge \ol{\h{u}} \wedge \ddbar \psi_L e^{-\psi_L}) + (-1)^N c_N p_{\ast}(\dbar\h{u} \wedge \overline{\dbar\h{u}} e^{-\psi_L}).
\end{align}
This formula make sense since $\dbar \h{u}$, $\partial^{\psi_L}\h{u}$, $\dbar\partial^{\psi_L}\h{u}$ and $\ddbar \psi_L$ are $L^2$-integrable with respect to $\omega_E$ and $e^{-\psi_L}$. Note that one have $ \ddbar \sum_{i\in I} a_i\log |z_i|^2 =0$ on $\cX^\circ$.
The next two lemmas help us to simplify the formula (\ref{ddbart}), which is the corollary of Hodge theory in section \ref{hodge theory}.

\begin{lem} \cite[Lemma $4.3$]{Bo09} \label{bo09}
    Let $u$ be an $(n,0)$-form on $\mathcal{X}$, representing a holomorphic section of $\mathcal{G}$, we can write
    \[
        \overline{\partial}u = \sum \eta^j \wedge dt_j.
    \]
    Then $\eta^j \wedge \omega$ are $\overline{\partial}$-exact on each fiber.
\end{lem}    
    \begin{proof}  
     We include the proof for the reader’s convenience. Since $u \wedge \omega$ is of bi-degree $(n+1,1)$, we can write locally
    \[
        u \wedge \omega = \sum u^j \wedge dt_j.
    \]
    The coefficients $u^j$ are not unique, but their restrictions on fibers are unique. Indeed, since $\sum u^j \wedge dt_j = 0$ implies $\sum u^j \wedge dt_j \wedge \h{d t_j} = u^j \wedge dt = 0$, the latter shows that $u^j$ vanishes when restricted to any fiber. Thus $u^j$ are well-defined global forms on any fiber. Moreover
    \[
        \sum \eta^j \wedge \omega \wedge dt_j = \overline{\partial}u \wedge \omega = \sum \overline{\partial}u^j \wedge dt_j,
    \]
    so $\sum (\eta^j \wedge \omega - \overline{\partial}u^j) \wedge dt_j = 0$. Again wedging with $\h{d t_j}$, we see that $\eta^j \wedge \omega = \overline{\partial}u^j$ on each fiber.
    \end{proof}

\begin{lem} \label{gap}
    Let $u$ be a holomorphic section of $\mathcal{G}$ on $D$ such that $\nabla u = 0$ at $t = 0$. Then $u$ can be represented by a $L^2$-form $\tilde{u}$ on $\cX^\circ$ such that $\overline{\partial}\tilde{u} = \eta^j \wedge dt_j$ for some $L^2$-form $\eta^j$ which are primitive on central open fiber $X^\circ_0$ and $\partial^{\psi_L}\tilde{u} = \mu^j \wedge dt_j$ for some $L^2$-form $\mu^j$ satisfying $\mu^j |_{X^\circ_0} = 0$.
\end{lem}

   \begin{proof}  
    We have $\partial^{\psi_L} u = \sum \mu^j \wedge dt_j$ and because $\nabla u(0) = 0$, therefore each $\mu^j_0 = \mu^j |_{X^\circ_0}$ is orthogonal to the space of $L^2$ holomorphic sections on $X^\circ_0$. We write $\overline{\partial} u = \sum \eta^j \wedge dt_j$, according the above Lemma \ref{bo09}, we know $\eta^j \wedge \omega = \overline{\partial}u^j$ on the center fiber. Since $\mu^j_0 - \overline{\partial}^*u^j$ is orthogonal to the space of $L^2$ holomorphic sections on $X^\circ_0$. The Hodge theory, especially Theorem \ref{Hodge1} shows that $\mu^j_0 - \overline{\partial}^*u^j$ is $\overline{\partial}^*$-exact, i.e., there exists a $\overline{\partial}$-closed $L^2$-form $\beta^j_0$ on $X^\circ_0$ such that $\overline{\partial}^* \beta^j_0 = \mu^j_0 - \overline{\partial}^* u^j$. We set $\gamma^j_0 = *(\beta^j_0 + u^j)$, then it easy to see
    \[
        \partial^{\psi_L} \gamma^j_0 = \mu^j_0,
    \]
    \[
        \overline{\partial}\gamma^j_0 \wedge \omega = \eta^j\wedge \omega.
\]
Let $\gamma^j$ be an arbitrary globally $L^2$ extension of $\gamma^j_0$ on $\cX^\circ$. Then $\tilde{u} = u - \sum \gamma^j \wedge dt_j$ is the representative we are looking for. We have $\overline{\partial}\tilde{u} = \overline{\partial}u - \sum \overline{\partial}\gamma^j \wedge dt_j = \sum(\eta^j - \overline{\partial}\gamma^j) \wedge dt_j$, thus on the center fiber
\[
(\eta^j - \overline{\partial}\gamma^j_0) \wedge \omega = \eta^j \wedge \omega - \overline{\partial}\gamma^j_0 \wedge \omega = 0.
\]
On the other hand, $\partial^{\psi_L} \tilde{u} = \partial^{\psi_L}u - \partial^{\psi_L}\gamma^j \wedge dt_j = \sum(\mu^j - \partial^{\psi_L}\gamma^j) \wedge dt_j$. When restrict it on $X^\circ_0$, we know $\mu^j - \partial^{\psi_L}\gamma^j = 0$ and therefore the claim is proved. 
\end{proof}

Note that we can assume $\dbar \gamma^j$, $\partial^{\psi_L}\gamma^j$ and $\dbar \partial^{\psi_L}\gamma^j$ are all $L^2$ integrable on $\cX^\circ$ with respect to $\omega_E$ and $e^{-\psi_L}$. Therefore, $\tilde{u}$ also satisfies these regularity properties. Indeed, since $\gamma^j_0$ possesses good regularity properties on the fibers, we work locally around $X_0^\circ$ (a tubular neighborhood), employing a coordinate cover and a partition of unity subordinate to it. On each coordinate patch, $\gamma^j_0$ can be extended locally by setting $\gamma^j(z,t):= \gamma_0^j(z)$. 
Using the partition of unity, we obtain a global extension over the tubular neighborhood of \( X_0^\circ \). To extend this to \( \mathcal{X}^\circ \), we multiply by a smooth function supported on this tubular neighborhood.

With the help of Lemma \ref{gap}, we can replace $u$ by $\w{u}$ in the definition of $T_u$,
$$ T_u = c_N p_{\ast}(\h{u} \wedge \ol{\h{u}} e^{-\psi_L}) = c_N p_{\ast}(\h{\w{u}} \wedge \ol{\h{\w{u}}} e^{-\psi_L}).
$$
Note that $T_u = T_{\tilde{u}}$ according to the construction of $\tilde{u}$. 
For ease of notation, we still denote by $u$ the section satisfying the conclusion in the above lemma. Then we can simplify the formula \eqref{ddbart} at the point $t=0$.
\begin{align} \label{ddbar}
\nonumber\partial \dbar T_u &= (-1)^N c_N p_{\ast}(\partial^{\psi_L} \h{u} \wedge \ol{\partial^{\psi_L}\h{u}} e^{-\psi_L}) -  c_N p_{\ast}(\h{u} \wedge \ol{\h{u}} \wedge \ddbar \psi_L e^{-\psi_L}) + (-1)^N c_N p_{\ast}(\dbar\h{u} \wedge \overline{\dbar\h{u}} e^{-\psi_L}) \\
&= -  c_N p_{\ast}(\h{u} \wedge \ol{\h{u}} \wedge \ddbar \psi_L e^{-\psi}) + (-1)^N c_N p_{\ast}(\dbar\h{u} \wedge \overline{\dbar\h{u}} e^{-\psi_L})
\end{align}

For further simplification, the next lemma would be useful.
\begin{lem}
Let $X$ be a complex manifold with complex dimension $n$, and if $\alpha$ is a form of bidegree $(n-1, 1)$. Then
$$ \sqrt{-1} c_{n-1} \alpha \wedge \bar{\alpha} = (\|\alpha\|^2 - \| \alpha\wedge \omega\|^2) dV_{\omega}.
$$
\end{lem}

\noindent Since we know $\eta$ is primitive on central open fiber $X_0^{\circ}$, we can further turn formula \eqref{ddbar} into,
$$ \sqrt{-1}\partial \dbar T_u = -  c_N p_{\ast}(\h{u} \wedge \ol{\h{u}} \wedge \sqrt{-1}\ddbar \psi_L e^{-\psi_L}) - \int_{X_0}\| \eta\|^2dV_t.
$$
Because $\sqrt{-1}\partial \dbar \psi_L$ is a positive current, and $c_N\h{u} \wedge\ol{\h{u}}$ be a strongly positive form. We obtain $c_N \h{u} \wedge \ol{\h{u}} \wedge \sqrt{-1}\ddbar \psi_L e^{-\psi_L}$ be a positive current. The push-forwards of positive current under proper morphism is again positive current. For details, we refer the reader to \cite[Chapter III. \S1]{bookJP}. It is obviously that $\sqrt{-1}\ddbar T_u$ is smooth $(m,m)$-form on $D$ and hence $c_N p_{\ast}(\h{u} \wedge \ol{\h{u}} \wedge \sqrt{-1}\ddbar \psi_L e^{-\psi_L})$ be the positive form. This formula means that $\sqrt{-1 }\partial \dbar T_u \leq 0$ and so $\mathcal{G}$ is positivity in the sense of Nakano.
Now the proof of Theorem \ref{Thm1} is completed.

\section{Some $L^2$ Hodge theory} \label{hodge theory}
We introduce a few results of $L^2$-Hodge theory for $L$-valued forms on a complete manifold endowed with a Poincar\'e type metric, following closely \cite{JCMP} and \cite{CGP}. 


Let $X$ be a $n$-dimensional compact K\"ahler manifold, and let $(L, h_L)$ be a line bundle  endowed with a (singular) metric $h_L= e^{-\psi_L}$ such that
\begin{enumerate} [label={\color{violet}(R.\arabic*)}]
\item \label{R1} $h_L$ has klt type singularities along divisor $E$ with simple normal crossing supports as Example \ref{3-ex};
\item\label{R2} $\omega_E$ be a complete K\"ahler metric on $X^{\circ} := X\backslash E$ with Poincar\'e singularities along $E$ as \eqref{pmetric};
\item \label{R3} Its Chern curvature satisfies $\displaystyle i\Theta_{h_L}(L) \geq 0$ in the sense of currents.
\end{enumerate}

\noindent We denote by $D'$ the $(1,0)$-part of the Chern connection on $(L,h_L)$. The famous Bochner--Kodaira--Nakano formula link up two Laplace $\Delta''= \dbar \dbar^*+\dbar^*\dbar$ and $\Delta'=D'D'^*+D'^*D'$ as follows,
\begin{equation} \label{bochner}
\Delta''=\Delta'+[i\Theta_{h_L} (L), \Lambda_{\omega_E}].
\end{equation}

\noindent Let us also recall the well-known fact that the self-adjoint operator
\[A:=[i\Theta_{h_L} (L), \Lambda_{\omega_E}]\] is semi-positive when acting on $(n,q)$ forms, for any $0\le q \le n$ as long as $i\Theta_{h_L}(L) \ge 0$.
An immediate consequence of \eqref{bochner} is that for a $L^2$-integrable form $u$ with values in $L$ of any type in the domains of $\Delta'$ and $\Delta''$, we have
\begin{equation}
\label{bochner2}
\|\dbar u\|^2_{L^2}+\|\dbar^*u\|^2_{L^2}=\|D' u\|^2_{L^2}+\|D'^*u\|^2_{L^2}+\int_{X^\circ}\langle Au,u\rangle dV_{\omega_E},
\end{equation}
where $\|\cdot\|_{L^2}$ (resp. $\langle, \rangle$) denotes the $L^2$-norm (resp. pointwise hermitian product) taken with respect to $(h_L, \omega_E)$. Let $\star:\Lambda^{p,q}T_{X^\circ}^*\to \Lambda^{n-q,n-p}T_{X^\circ}^*$ be the Hodge star with respect to $\omega_E$;

\begin{thm} \cite[Theorem 3.6]{CGP}\label{Hodge1}
Let $X$ be a $n$-dimensional compact K\"ahler manifold, and let $(L, h_L)$ be a line bundle  endowed with a (singular) metric $h_L= e^{-\psi_L}$ such that the assumption \ref{R1}, \ref{R2}, \ref{R3} are satisfied, then we have the following Hodge decomposition
\begin{equation*}
L^2_{n, 1}(X^\circ, L)= {\mathcal H}_{n,1}(X^\circ, L)\oplus
\im \bar{\partial} \oplus \im \bar{\partial}^*,
\end{equation*}
here ${\mathcal H}_{n,1}(X^\circ, L)$ is the space of $L^2$ $\Delta^{''}$-harmonic $(n,1)$-forms.	
\end{thm}

The main goal of this section is to establish the following decomposition theorem, which is analogous to the corresponding result in \cite{CGP}. For the reader's convenience, we will sketch the proof and emphasize the differences.
Let $X$ be a $n$-dimensional compact K\"ahler manifold, and let $(L, h_L)$ be a line bundle  endowed with a (singular) metric $h_L= e^{-\psi_L}$ such that
\begin{enumerate} [label={\color{violet}(S.\arabic*)}]
\item \label{S1} $h_L$ has logarithmic type singularities along divisor $E$ with simple normal crossing supports, i.e., its local weights $\psi_L$ on $\Omega$ can be written as
\begin{equation*}
\psi_L \equiv - \sum_{I} b_I\log \big(\phi_I(z)-\log (\prod_{i\in I}|z_i|^{2k_i}) \big)
\end{equation*}
modulo $\mathcal C^\infty$ functions, where $b_I$ are positive real numbers satisfying that $b_I < 1$ for all $I$, all $k_i$ are non-negative integers and $(\phi_I)_I$ are smooth functions on $\Omega$. The set of indexes in the sum coincides with the non-empty subsets of $\{1,\dots, r\}$.
\item\label{S2} $\omega_E$ be a complete K\"ahler metric on $X^{\circ} := X\backslash E$ with Poincar\'e singularities along $E$ as \eqref{pmetric};
\item \label{S3} Its Chern curvature satisfies $\displaystyle i\Theta_{h_L}(L) \geq 0$ in the sense of currents.
\end{enumerate}

\begin{thm}\label{Hodge}
Let $X$ be a $n$-dimensional compact K\"ahler manifold, and let $(L, h_L)$ be a line bundle  endowed with a (singular) metric $h_L= e^{-\psi_L}$ such that the assumption \ref{S1}, \ref{S2}, \ref{S3} are satisfied, then we have the following Hodge decomposition
\begin{equation*}
L^2_{n, 1}(X^\circ, L)= {\mathcal H}_{n,1}(X^\circ, L)\oplus
\im \bar{\partial} \oplus \im \bar{\partial}^*,
\end{equation*}
here ${\mathcal H}_{n,1}(X^\circ, L)$ is the space of $L^2$ $\Delta^{''}$-harmonic $(n,1)$-forms.	
\end{thm}

Firstly we need a lemma that assures we can use the compact support approximation.
\begin{lem} \cite[Lemma 3.7]{CGP} \label{cutoff}
There exists a family of smooth functions $(\mu_\ep)_{\ep> 0}$ with the following properties.
	\begin{enumerate}
		\item[\rm (1)] For each $\ep> 0$, the function $\mu_\ep$ has compact support in $X^\circ$, and $0\leq \mu_\ep\leq 1$;
		\item[\rm (2)] The sets $(\mu_\ep= 1)$ are providing an exhaustion of $X^\circ$;
		\item[\rm (3)] There exists a positive constant $C> 0$ independent of $\ep$ such that
		we have $$\displaystyle \sup_{X^\circ}\left(|\partial \mu_\ep|_{\omega_E}^2+ |\ddbar \mu_\ep|^2_{\omega_E}\right)\leq C.$$
	\end{enumerate}
\end{lem}

\noindent We also need the local Poincar\'e type inequality for the $\dbar$-operator when acting on $L$-valued $(p,0)$-forms. The next proposition is very important for our proposal. This is about the Poincar\'e type inequality involving the metric of a line bundle, and the metric $\psi_L$ is no longer with analytic singularities (as in \cite{JCMP}) but with Poincar\'e type singularities. Therefore we need to reprove it.

\begin{prop}\label{poinc}
Let $\displaystyle (\Omega_j)_{j=1,\dots, N}$ be a finite union of coordinate sets of $ X$ covering $E$, and let $\ U$ be any open subset contained in their union. Let $\tau$ be a $(p, 0)$-form with compact support in a set $U\setminus E\subset X$ and values in $(L, h_L)$. Then we have
\begin{equation} \label{eq57}
 \int_U|\tau|_{\omega_E}^2e^{-\psi_L}dV_{\omega_E}\leq C \int_U|\dbar\tau|^2_{\omega_E}e^{-\psi_L}dV_{\omega_E}
\end{equation}
where $C$ is a positive numerical constant.
\end{prop}

\begin{proof}
We consider the restriction of the form $\tau$ to some coordinate open set $\Omega$ whose intersection with $E$ is of type $z_1\dots z_r= 0$. It can be written as sum of forms of type $\tau_I:= f_Idz_I$ and assume that $I\cap \{1,\dots, r\}= \{1,\dots, p\}$ for some $p$. We have
\begin{equation}\label{d222}
|\tau_I|_g^2e^{-\psi_L}dV_g = \frac{|f_I|^2e^{-\psi_L}}{\prod_{\alpha=p+1}^r|z_\alpha|^2\log^2|z_\alpha|^2}
d\lambda.
\end{equation}

\noindent In \eqref{d222} we denote by $g$ the model Poincar\'e metric
$$\sum_{i=1}^r\frac{\sqrt{-1}dz_i\wedge d\bar z_i}{|z_i|^2\log^2|z_i|^2}+ \sum_{i=r+1}^n{\sqrt{-1}dz_i\wedge d\bar z_i},$$
\noindent and $d \lambda = \frac{\omega^n}{n!}$, where $\omega = \sum^n_{i=1} \sqrt{-1}d z_i \wedge d \bar z_i$. By our assumption \ref{R2}, we know
$$
e^{-\psi_L} \equiv
\prod_{I}\left(\phi_I(z)-\log \big(\prod_{i\in I}|z_i|^{2k_i}\big)\right)^{b_I}
$$
modulo $\mathcal{C}^{\infty}$ function. Now for the $\dbar \tau_I$ we have
\begin{align} \label{d333}
  |\dbar \tau_I|_g^2e^{-\psi_L}dV_g= & \sum_{i=1}^r\left|\frac{\partial f_I}{\partial \bar z_i}\right|^2\frac{|z_i|^2\log^2|z_i|^2e^{-\psi_L}}{\prod_{\alpha=p+1}^r|z_\alpha|^2\log^2|z_\alpha|^2}d\lambda
  +  \sum_{i=r+1}^n\left|\frac{\partial f_I}{\partial \bar z_i}\right|^2\frac{e^{-\psi_L}}{\prod_{\alpha=p+1}^r|z_\alpha|^2\log^2|z_\alpha|^2} d\lambda.
\end{align}

\noindent To prove the claim, we can reduce it to the one-dimension case by Fubini's theorem, i.e., unit disk $\mathbb{D}$ in complex plane $\mathbb{C}$. Combining \eqref{d222} and \eqref{d333}, it is easy to see that the next two inequalities are all we need, this is
\begin{equation}\label{d1}
 \int_{\mathbb D}|f|^2  (-\log |z|)^b d\lambda(z)   \leq C \int_{\mathbb D}\left| \frac{\partial f }{\partial \bar z}\right|^2 |z|^{2} (-\log |z|)^{b+2} d\lambda(z),
\end{equation}
as well as
\begin{equation}\label{d2}
\int_{\mathbb D}|f|^2 \frac{(-\log |z|)^{b} d\lambda(z)}{|z|^{2}\log^2 |z|} \leq C \int_{\mathbb D}\left| \frac{\partial f }{\partial \bar z}\right|^2 (-\log |z|)^{b} d\lambda(z).
\end{equation}
Here $b$ is positive real number and we can arrange $f$ has small compact support, i.e., supp$f\subset \{ z~|~ 0<|z|<\epsilon \}$ where $\epsilon$ small enough. Furthermore, it's better to reduce this integration to real line $\mathbb{R}$, and hence we consider the Fourier series
\begin{equation*}
f=\sum_{k\in \mathbb Z}a_k(t)e^{\sqrt{-1}k\theta}
\end{equation*}
of $f$, thus we have
\begin{equation*}
\frac{\partial f }{\dbar z}= \sum_{k\in \mathbb Z}\left(a_k'(t)- \frac{k}{t}a_k(t)\right)e^{\sqrt{-1}k\theta}.
\end{equation*}
The identity
\begin{equation*}
t^k\frac{d}{dt}\left(\frac{a_k}{t^k}\right)= a_k'(t)- \frac{k}{t}a_k(t)
\end{equation*}
reduces the proof of our statement to the next inequalities. We have
$$\int_{\mathbb D}|f|^2  (-\log |z|)^b d\lambda(z) = \sum_k \int_0^1  |a_k(t)|^2 (-\log t)^b t dt $$
and
$$  \int_{\mathbb D}\left| \frac{\partial f }{\partial \bar z}\right|^2 |z|^{2} (-\log |z|)^{b+2} d\lambda(z) =
\sum_k \int_0^1 |t^k \frac{d}{dt}(\frac{a_k}{t^{k}})|^2 t^3 (-\log t)^{b+2} dt.
$$
\noindent Let $b_k = \frac{a_k}{t^k}$, then we turn the above two equalities into
$$\int_{\mathbb D}|f|^2  (-\log |z|)^b d\lambda(z) = \sum_k \int_0^1  |b_k(t)|^2 (-\log t)^b t^{2k+1} dt $$
and
$$  \int_{\mathbb D}\left| \frac{\partial f }{\partial \bar z}\right|^2 |z|^{2} (-\log |z|)^{b+2} d\lambda(z) =
\sum_k \int_0^1 |b'_k(t)|^2 t^{2k+3} (-\log t)^{b+2} dt.
$$
\noindent So the inequality \eqref{d1} is equivalent to the next inequality
\begin{equation} \label{f1}
\sum_k \int_0^1  |b_k(t)|^2 (-\log t)^b t^{2k+1} dt  \leq C  \sum_k \int_0^1 |b'_k(t)|^2 t^{2k+3} (-\log t)^{b+2} dt.
\end{equation}
\noindent Here $k \in \mathbb{Z}, b\in \mathbb{R}^+$ and $b_k(t)$ has small compact support. In fact the inequality \eqref{d2} is also come from \eqref{f1}. Now we prove the inequality \eqref{f1}, according to integration by part, we have
\begin{align}
\nonumber  \int_0^1  |b_k(t)|^2 (-\log t)^b t^{2k+1} dt \underline{}=& -  \int_0^1  |b_k(t)|^2 (-\log t)^{b+2} t^{2k+2} d(\frac{1}{\log t}) \\
\nonumber =& \int_0^1 \frac{1}{\log t} d(|b_k|^2 t^{2k+2} (-\log t)^{b+2}).
\end{align}
\noindent As a consequence, we obtain
 \begin{align} \label{f2}
&\int_0^1  |b_k(t)|^2 (-\log t)^b t^{2k+1} dt \\
\nonumber =& - \int_0^1(b'_k \bar b_k + b_k \bar b'_k) t^{2k+2}(-\log t)^{b+1} dt + \int_0^1 \frac{|b_k|^2 t^{2k+1}}{\log t} \big[(2k+2)(-\log t)^{b+2}- (b+2)(-\log t)^{b+1}\big] dt.
\end{align}
\noindent Now we divide it into three cases by the sign of $(2k+2)$:

\noindent $(1)$ if $2k+2 =0$, then the second term in \eqref{f2} above is $ \int_0^1 (b+2)|b_k|^2 t^{2k+1}(-\log t)^b dt $. Thus we have
\begin{align} \label{f3}
 \nonumber &\int_0^1 (b+1)|b_k|^2 t^{2k+1}(-\log t)^b = \int_0^1(b'_k \bar b_k + b_k \bar b'_k) t^{2k+2}(-\log t)^{b+1} dt \\
  &\leq \sqrt{\int_0^1 |b_k|^2 t^{2k+1}(-\log t)^b dt} \sqrt{\int_0^1 |b'_k|^2 t^{2k+3}(-\log t)^{b+2} dt}.
\end{align}

\noindent Therefore we have the desired inequality
$$\int_0^1 |b_k|^2 t^{2k+1}(-\log t)^b dt \leq C \int_0^1 |b'_k|^2 t^{2k+3}(-\log t)^{b+2} dt.
$$

\noindent $(2)$ if $ 2k+2> 0 $, when $t$ is small enough, we have
$$ (2k+2)(-\log t)^{b+2} - (b+2)(-\log t)^{b+1} \geq (b+2)(-\log t)^{b+1}.
$$
Then the second term in \eqref{f2} is small than
$$
 \int_0^1 \frac{|b_k|^2 t^{2k+1}}{\log t} \big[(2k+2)(-\log t)^{b+2}- (b+2)(-\log t)^{b+1}\big] dt \leq  -(b+2) \int_0^1 |b_k|^2 t^{2k+1} (-\log t)^b.
$$
A similar calculation yields inequality \eqref{f1} holds.

\noindent $(3)$ if $2k+2 < 0$ and also make $t$ small enough, we have
$$ (2k+2)(-\log t)^{b+2} - (b+2)(-\log t)^{b+1} \leq -(b+2)(-\log t)^{b+1}.
$$
Hence the second term in \eqref{f2} is bigger than
$$ \int_0^1 \frac{|b_k|^2 t^{2k+1}}{\log t} \big[(2k+2)(-\log t)^{b+2}- (b+2)(-\log t)^{b+1}\big] dt \geq (b+2) \int_0^1 |b_k|^2 t^{2k+1}(-\log t)^{b}.
$$
Inequality \eqref{f2} now become to
$$  \int_0^1 |b_k|^2 t^{2k+1} (-\log t)^{b} dt \geq - \int_0^1(b'_k \bar b_k + b_k \bar b'_k) t^{2k+2}(-\log t)^{b+1} dt + \int_0^1 (b+2) |b_k|^2 t^{2k+1} (-\log t)^{b} dt.
$$
By Cauchy inequality as \eqref{f3} above, we can also easily deduce that inequality \eqref{f1} is valid in this case.

In summary, the proof of the proposition is completed.
\end{proof}

\noindent Once the local Poincar\'e inequality was established, as a corollary, we can obtain the following global Poincar\'e inequality for $\dbar$. In \cite{CGP}, the authors introduce for any integer $0\le p \le n$ the space
\begin{equation} \label{Hp}
H^{(p)}:=\{F\in  H^0(X^\circ, \Omega_{X^\circ}^{p}\otimes L)\cap L^2;  \,\, \int_{X^\circ}\langle A\star F, \star F\rangle dV_{\omega_E}=0\}
\end{equation}
and we can observe by Bochner formula that for a $L^2$ integrable, $L$-valued $(p,0)$-form $F$, one has
\begin{equation} \label{hp}
 \Delta''(\star F)=0 \Longleftrightarrow \Delta'(\star F)=0 \text{ and } \int_{X^\circ}\langle A\star F, \star F\rangle dV_{\omega_E}=0 \Longleftrightarrow F\in H^{(p)}  .
\end{equation}

\begin{prop}\label{cor:2} \cite[Proposition 3.9]{CGP}
	Let $p\leq n$ be an integer.  There exists a positive constant $C> 0$ such
	that the following inequality holds
	\begin{equation}\label{eq3233}
	\int_{X^\circ}|u|^2_{\omega_E}e^{-\psi_L}dV_{\omega_E}\leq C \left(\int_{X^\circ}|\dbar u|^2_{\omega_E}e^{-\psi_L}dV_{\omega_E} +
	\int_{X^\circ} \langle A \star u , \star u \rangle d V_{\omega_E}\right)
	\end{equation}
	for any $L$-valued form $u$ of type $(p, 0)$ which belongs to the domain of $\dbar$ and which is orthogonal to the  space $H^{(p)}$ defined by \eqref{Hp}. 	Here $\star$ is the Hodge star operator with respect to the metric $\omega_E$.
\end{prop}

\noindent We have the following direct consequences of Proposition \ref{cor:2}.

\begin{cor} \cite[Proposition 3.10]{CGP}
	There exists a positive constant $C> 0$ such
	that the following inequality holds
	\begin{equation}\label{eq32334}
	\int_{X^\circ}|u|^2_{\omega_E}e^{-\psi_L}dV_{\omega_E}\leq C \left(\int_{X^\circ}|\dbar u|^2_{\omega_E}e^{-\psi_L}dV_{\omega_E}\right)
	\end{equation}
	for any $L$-valued form $u$ of type $(n, 0)$ which belongs to the domain of $\dbar$ and which is orthogonal to the kernel of $\dbar$.
\end{cor}

\begin{proof} This follows immediately from Proposition \ref{cor:2} combined with the observation that the
  curvature operator $A$ is equal to zero in bi-degree $(n,0)$.
\end{proof}

\noindent The next statement shows that in bi-degree $(n,p)$ the image of the operator $\dbar^\star$ is closed.

\begin{cor} \cite[Proposition 3.11]{CGP} \label{star closed}
There exists a positive constant $C> 0$ such that the following holds.
	Let $v$ be a $L$-valued form of type $(n,p)$. We assume that $v$ is $L^2$, in the domain of $\dbar$ and orthogonal
	to the kernel of the operator $\dbar^\star$. Then we have
	\begin{equation}\label{eq62}
	\int_{X^\circ}|v|^2_{\omega_E}e^{-\psi_L}dV_{\omega_E}\leq C \int_{X^\circ}|\dbar^\star v|^2_{\omega_E}e^{-\psi_L}dV_{\omega_E}.
	\end{equation}
\end{cor}

\begin{proof}
Let us first observe that the Hodge star $u:= \star v$, of type $(n-p,0)$, is orthogonal to $H^{(n-p)}$. Actually, let us pick $F\in H^{(n-p)}$, it follows from \eqref{hp} that we have $\displaystyle \dbar^\star(\star F)= 0$. In other words, $\star F \in \ker \dbar^\star$.
We thus have
$$\int_{X^\circ} \langle u, F\rangle   d V_{\omega_E} = \int_{X^\circ} \langle v, \star F \rangle  d V_{\omega_E} =0.$$
Applying Bochner formula \eqref{bochner2} to $v$ and using the facts that $\dbar  v=0$ (since $v$ is orthogonal to $\ker \dbar^*$) and that $\dbar^*u=0$ for degree reasons, we get
\begin{equation}\label{eq63}
\|\dbar^*v\|^2_{L^2}=\|\dbar u\|^2_{L^2}+\int_{X^\circ}\langle A \star u, \star u\rangle dV_{\omega_E}
\end{equation}
This proves the corollary by applying Proposition \ref{cor:2}.
\end{proof}

Now we can prove Theorem \ref{Hodge}.
\begin{proof}[Proof of Theorem~\ref{Hodge}]
In the context of complete manifolds, one has the following decomposition
\begin{equation*}\label{eq33}
L^2_{n, 1}(X^\circ, L)= {\mathcal H}_{n,1}(X^\circ, L)\oplus \overline{ \im \dbar}\oplus \overline{ \im \dbar^\star}.
\end{equation*}
The reader can consult the reference \cite[chapter VIII, pages
367-370]{bookJP} for details. We also know that the adjoint $\dbar^\star$ and $D^{\prime\star}$ in the sense of von Neumann coincide with the formal
adjoint of $\dbar$ and $D^\prime$ respectively. It remains to show that the range of the $\dbar$ and $\dbar^\star$ operators are closed concerning the $L^2$ topology. We will utilize the inequalities \eqref{eq32334} and \eqref{eq62}, the former shows that the image of $\dbar$ is closed, and the latter does the same for $\dbar^\star$. Indeed, suppose that there exists a sequence ${\dbar u_j}\rightarrow u$ in the $L^2$ space. Since $\dbar u_j = \dbar(u_j^1 + u_j^2)= \dbar u_j^1$ (here we settle $u_j^1 \perp \ker \dbar$ and $u_j^2 \in \ker \dbar$), we can assume each $u_j$ are orthogonal to the kernel space of $\dbar$. By \eqref{eq32334}, we obtain
$$ \int_{X^\circ}|u_j-u_k|^2_{\omega_E}e^{-\psi_L}dV_{\omega_E}\leq C \left(\int_{X^\circ}|\dbar u_j -\dbar u_k|^2_{\omega_E}e^{-\psi_L}dV_{\omega_E}\right),
$$
therefore ${u_j}$ is Cauchy sequence. On the other hand, we know $\dbar$ is a closed operator. This yields that $u \in \im \dbar$. Similarly, we can prove the image of $\dbar^{\ast}$ is closed too.
\end{proof}

\end{document}